\documentclass[oneside,reqno]{amsart}

\usepackage{hyperref}
\usepackage[totalheight=20 true cm, totalwidth=15 true cm]{geometry}

\usepackage[shortcuts]{extdash}

\usepackage{xypic}

\usepackage{enumerate}
\usepackage{amsthm}
\usepackage{amsmath}
\usepackage{amssymb}
\usepackage{amsfonts}
\usepackage{dsfont}
\usepackage{array}
\usepackage{tabularx}
\usepackage{graphicx}

\newtheorem{theo}{Theorem}[section]
\newtheorem{lemma}[theo]{Lemma}
\newtheorem{prop}[theo]{Proposition}
\newtheorem{cor}[theo]{Corollary}

\theoremstyle{definition}

\newcommand{\f}{\phi}

\newcommand{\spec}{\operatorname{Spec}}

\newcommand{\Aut}{\operatorname{Aut}}

\newcommand{\A}{\mathbb{A}}

\makeatletter
\@namedef{subjclassname@2020}{\textup{2020} Mathematics Subject Classification}
\makeatother

\title{A remark on torsors for affine group schemes}
\author{Michael Wibmer}

\address{Michael Wibmer, Institute of Analysis and Number Therory, Graz University of Technology, Kopernikusgasse~24, 8010 Graz, Austria, \url{https://sites.google.com/view/wibmer}}
\email{wibmer@math.tugraz.at}

\thanks{This work was supported by the NSF grants DMS-1760212, DMS-1760413, DMS-1760448 and the Lise Meitner grant M 2582-N32 of the Austrian Science Fund FWF}

\date{\today}

\begin{document}
\maketitle

\begin{abstract}
We present an elementary proof of the fact that every torsor for an affine group scheme over an algebraically closed field is trivial. This is related to the uniqueness of fibre functors on neutral tannakian categories.

%This implies that any two fibre functors on a neutral tannakian category over an algebraically closed field are isomorphic.
\end{abstract}

\section{Introduction}

Clearly, every torsor for an affine group scheme of finite type over an algebraically closed field is trivial. However, it is not clear if this also holds without the finite type assumption. This question was raised in \cite{Deninger:ARemarkOnTheStructureOfTorsorsUnderAnaffineGroupScheme}, where some partial positive results were obtained: When the affine group scheme $G$ is written as a projective limit of affine group schemes of finite type over an index set $I$, then all $G$-torsors are trivial in the following two cases: $I$ is countable or the cardinality of the algebraically closed base field is strictly larger than the cardinality of $I$. The main result discussed in this short note is the following:

\begin{theo} \label{theo: main intro}
	Let $G$ be an affine group scheme over an algebraically closed field $k$ and let $X$ be a $G$-torsor. Then $X$ is trivial, i.e., $X(k)\neq\emptyset$.	
\end{theo}

As an application one obtains:

\begin{cor} \label{cor: intro uniquenes of fibre functor}
	Any two neutral fibre functors on a neutral tannakian category over an algebraically closed field are isomorphic.
\end{cor}

We note that a different proof of Corollary \ref{cor: intro uniquenes of fibre functor} is outlined in \cite{Deligne:LetterToAVasiu}. In fact, Corollary \ref{cor: intro uniquenes of fibre functor} implies Theorem \ref{theo: main intro}, because by \cite[Theorem 3.2 (b)]{DeligneMilne:TannakianCategories}, for $\omega$ the forgetful functor on the category $\operatorname{Rep}(G)$ of finite dimensional $k$-linear representations of $G$, the functor $\eta\rightsquigarrow \underline{\operatorname{Isom}}^\otimes(\omega,\eta)$ is an equivalence of categories between the category of neutral fibre functors on $\operatorname{Rep}(G)$ and the category of $G$-torsors. Cf. Remark 6.4.4 in \cite{Coulembier:TannakianCategoriesInPositiveCharacteristic}.

In this note we present an elementary proof of Theorem \ref{theo: main intro} that relies on a general principle guaranteeing the non-emptiness of a projective limit. We also show that (over an arbitrary base field) any $G$-torsor can be written as a projective limit of affine $G$-spaces of finite type that are torsors for quotient groups of $G$ (Proposition~\ref{prop: torsor is projective limit}).

\section{Proofs}

We begin by fixing our notation. Throughout $k$ is a field; our base field. All schemes (including group schemes), products, tensor products and morphisms are assumed to be over $k$ unless the contrary is indicated.

We will often identify a scheme $X$ with its functor of points $R\rightsquigarrow X(R)$ from the category of $k$-algebras to the category of sets. For an affine scheme $X$ we denote with $k[X]$ its $k$-algebra of global sections.

Let $G$ be an affine group scheme. By a \emph{closed subgroup} of $G$ we mean a closed subgroup scheme of $G$. A \emph{$G$-space} is a scheme $X$ together with a $G$-action (from the right) $X\times G\to X,\ (x,g)\mapsto x.g$. A morphism of $G$-spaces is a $G$-equivariant morphism of schemes.
A \emph{$G$-torsor} is a $G$-space $X$ such that $X\times G\to X\times X,\ (x,g)\mapsto (x,x.g)$ is an isomorphism. For an affine $G$-space $X$ the \emph{centralizer} $C_G(X)$ is defined by 
$$C_G(X)(R)=\{g\in G(R)|\ x.g=x \ \forall \ x\in X(R') \text{ and all $R$-algebras $R'$}\}.$$ 
Then $C_G(X)$ is a normal closed subgroup of $G$ (\cite[Chapter II, Theorem 3.6 c)]{DemazureGabriel:GroupesAlgebriques}) and $X$ is a $G/C_G(X)$-space. Following \cite[Def. 5.5.]{Milne:AlgebraicGroupsTheTheoryOfGroupSchemesOfFiniteTypeOverAField}
we call a morphism $G\to H$ of affine group schemes a \emph{quotient map} if it is faithfully flat (equivalently, the dual map $k[H]\to k[G]$ is injective \cite[Section 14]{Waterhouse:IntroductiontoAffineGroupSchemes}).

\medskip

Let us first sketch the proof of Theorem \ref{theo: main intro}. We are given a $G$-torsor $X$ and we would like to show that $X(k)\neq \emptyset$. We can write $X$ as a projective limit $X=\varprojlim X_i$ of $G$-spaces $X_i$ of finite type. So $X(k)=\varprojlim X_i(k)$.  Since we assume $k$ to be algebraically closed, the $X_i(k)$'s are non-empty. However, a projective limit of non-empty sets may well be empty. A standard condition to guarantee the non-emptyness of a projective limit of sets is that the sets are compact Hausdorff topological spaces with continuous transition maps (\cite[Prop. 1.1.4]{RibesZalesskii:ProfiniteGroups}.) Unfortunately, the $X_i(k)$'s equipped with the Zariski topology are not Hausdorff and so another approach is needed. The following lemma (see \cite[Prop. 2.7]{HochschildMostow:RepresentationsAndRepresentativeFunctionsOfLieGroups} or \cite[Theorem 2.1]{Stewart:ConjugacyTheoremsForAClassOfLocallyFiniteLieAlgebras}) provides a more refined criterion to show that a projective limit is non-empty.

\begin{lemma} \label{lemma: projective limit nonempty}
	Let $I$ be a directed set and let $((X_i)_{i\in I}, (\varphi_{i,j})_{i\leq j})$ be a projective system of topological spaces. If the $X_i$'s are non-empty compact T1 spaces and the $\varphi_{i,j}$'s are closed maps, then $\varprojlim X_i$ is non-empty.
\end{lemma}

Returning to the above discussion, the $X_i(k)$'s are compact T1 spaces with respect to the Zariski topology. However, the transition maps need not be closed and so Lemma \ref{lemma: projective limit nonempty} cannot be applied directly. A different topology is needed. We first show that the $X_i$'s can be chosen in such a way that $X_i$ is a $G/C_G(X_i)$-torsor. Using this property, we show that the subsets of $X_i(k)$ that are finite unions of orbits of the form $x.H(k)$ with $x\in X_i(k)$ and $H$ a closed subgroup of $G$, are the closed subsets of a topology on $X_i(k)$; the \emph{orbit topology}. With respect to the orbit topology $X_i(k)$ is a compact T1 space and the transition maps are continuous and closed. Thus Lemma \ref{lemma: projective limit nonempty} applied to the projective system of the $X_i(k)$'s equipped with the orbit topology yields Theorem~\ref{theo: main intro}.

\medskip

To make the above sketch precise, we will use the action of $G$ on $k[X]$. Let $G$ be an affine group scheme and $X$ an affine $G$-space. The $G$-action $X\times G\to X$ induces a functorial (left) action of $G$ on $k[X]$. For any $k$-algebra $R$, the group $G(R)$ acts on $k[G]\otimes R$ by $R$-algebra automorphisms. Identifying $k[X]\otimes R$ with the set of morphisms from $X_R$ to $\A^1_R$, the action of $g\in G(R)$ on $f\in k[X]\otimes R$ is given by $g(f)(x)=f(x.g)$ for $x\in X(R')$ and any $R$-algebra $R'$. The invariant ring under this action is $$k[X]^G=\{f\in k[X]|\ g(f\otimes 1)=f\otimes 1 \ \forall\ g\in G(R) \text{ and any $k$-algebra $R$}\}.$$
It is a $k$-subalgebra of $k[X]$. Note that for a normal closed subgroup $N$ of an affine group scheme $G$ acting via right-multiplication on $G$, we have $k[G]^N=k[G/N]$. See e.g. \cite[Section 16.3]{Waterhouse:IntroductiontoAffineGroupSchemes}.

\begin{lemma} \label{lemma: diagonal invariants}
	Let $X\times G\to X$ and $Y\times H\to Y$ be actions of affine group schemes on affine schemes. With respect to the diagonal action of $G\times H$ on $X\times Y$ we have
	$$k[X\times Y]^{G\times H}=k[X]^G\otimes k[Y]^H.$$
\end{lemma}
\begin{proof}
	Note that for a $k$-algebra $R$ and $g\in G(R)$, $h\in H(R)$, the action of $(g,h)$ is given by $$(g,h)\colon k[X\times Y]\otimes R=(k[X]\otimes R)\otimes_R (k[Y]\otimes R)\xrightarrow{g\otimes h} (k[X]\otimes R)\otimes_R (k[Y]\otimes R)=k[X\times Y]\otimes R.$$ So the inclusion 
	$k[X]^G\otimes k[Y]^H\subseteq k[X\times Y]^{G\times H}$ is clear.
	
	Conversely, assume that $\sum a_i\otimes b_i\in (k[X]\otimes k[Y])^{G\times H}$. We may assume that the $a_i$'s are $k$-linearly independent. For any $k$-algebra $R$ and $h\in H(R)$ we have
	$$(1,h)\big(\sum a_i\otimes b_i\otimes 1\big)=\sum a_i\otimes h(b_i\otimes 1)=\sum a_i\otimes b_i\otimes 1\in k[X]\otimes k[Y]\otimes R$$
	As the $a_i$'s are $k$-linearly independent we can conclude that $h(b_i\otimes 1)=b_i\otimes 1$, i.e., $b_i\in k[Y]^H$. 
	
	Now, assuming that the $b_i$'s are $k$-linearly independent, a similar argument shows that the $a_i$'s must lie in $k[X]^G$. Thus $\sum a_i\otimes b_i\in k[X]^G\otimes k[Y]^H$.
\end{proof}

It is well known (see e.g. \cite[Section 3.3]{Waterhouse:IntroductiontoAffineGroupSchemes}) that every affine group scheme is a projective limit of affine algebraic groups. The following proposition shows that a similar statement is true for torsors. 

\begin{prop} \label{prop: torsor is projective limit}
	Let $G$ be an affine group scheme and let $X$ be an affine $G$-torsor. Then $X$ can be written as a projective limit $X=\varprojlim_{i\in I} X_i$ of affine $G$-spaces $X_i$ of finite type such that every $X_i$ is a $G/C_G(X_i)$-torsor.
\end{prop}
\begin{proof}	
%	There is a functorial (left) action of $G$ on $k[X]$ derived from the $G$-action $X\times G\to X$. For any $k$-algebra $R$, the group $G(R)$ acts on $k[G]\otimes_k R$ by $R$-algebra automorphisms. Identifying $k[X]\otimes_kR$ with the set of morphisms from $X_R$ to $\A^1_R$, the action of $g\in G(R)$ on $f\in k[X]\otimes_k R$ is given by $g(f)(x)=f(xg)$ for $x\in X(R')$ and any $R$-algebra $R'$. 
	If an abstract group $G$ acts (from the right) on a set $X$ such that $X$ is a $G$-torsor, then for any normal subgroup $N$ of $G$ the set $X/N$ of $N$-orbits in $X$ is a $G/N$-torsor under the action $X/N\times G/N\to X/N,\ (x.N,gN)\mapsto x.g.N$. This is the idea for the construction of the $X_i$'s. However, to avoid a discussion of the existence of $X/N$ (as an affine scheme) in our context, we will mainly work with the invariant rings.	
	
	Let $N$ be a normal closed subgroup of $G$ such that $G/N$ is algebraic (i.e., of finite type). Then $N$ acts (form the right) on $X$ and on $G$.	Let $\rho\colon k[X]\to k[X]\otimes k[G]$ be the dual of the action $X\times G\to X$. We claim that $\rho$ restricts to a map $k[X]^N\to k[X]^N\otimes k[G]^N$.

	We have a (right) action of $N\times N$ on $X\times G$ given by $(x,g).(n_1,n_2)=(x.n_1,gn_2)$ for $x\in X(R),\ g\in G(R),\ n_1,n_2\in N(R)$ and $R$ a $k$-algebra. According to Lemma \ref{lemma: diagonal invariants}, the invariants $k[X\times G]^{N\times N}$ with respect to this action are equal to $k[X]^N\otimes k[G]^N$. It thus suffices to show that $\rho$ maps an $N$-invariant $f\in k[X]$ to an $(N\times N)$-invariant, i.e., we have to show that $f(x.g)=f(x.n_1gn_2)$ for $n_1,n_2\in N(R)$, $x\in X(R')$, $g\in G(R')$, where $R$ is a $k$-algebra and $R'$ an $R$\=/algebra. But since $g^{-1}n_1gn_2\in N(R')$, we have $f(x.g)=f(x.g.g^{-1}n_1gn_2)=f(x.n_1gn_2)$ by the $N$-invariance of $f$.
%	
%	
%	
%	In other words, if $f(x.n)=f(x)$ for any $n\in N(R)$ and $x\in X(R')$ with $R'$ an $R$-algebra, then This is true because $x.n_1gn_2=x.gg^{-1}n_1gn_2$ and $g^{-1}n_1gn_2\in N(R')$ since $N$ is normal in $G$. ??
	
		Thus $\rho$ restricts to a well-defined map $\rho_N\colon k[X]^N\to k[X]^N\otimes k[G]^N$.
	Setting $X_N=\spec(k[X]^N)$ we thus have an action $X_N\times G/N\to X_N$ of $G/N$ on $X_N$. We claim that $X_N$ is a $G/N$-torsor.
	
	The dual $\psi\colon k[X]\otimes k[X]\to k[X]\otimes k[G]$ of the isomorphism $X\times G\to X\times X,\ (x,g)\mapsto (x,x.g)$ is an isomorphism. Therefore, the dual $\psi_N\colon k[X]^N\otimes k[X]^N\to k[X]^N\otimes k[G]^N$ of $X_N\times G/N\to X_N\times X_N,\ (x,g)\mapsto (x,x.g)$ is at least injective.
	
	 To see that $\psi_N$ is surjective, we consider the $(N\times N)$-invariants on both sides of the isomorphism $\psi$ (Lemma \ref{lemma: diagonal invariants}). Note however, that $\psi$ is not $(N\times N)$-equivariant. 
	%Indeed, $\f((x,g).(n_1,n_2))=\f(x.n_1,gn_2)=(x.n_1,x.n_1g n_2)$, which is in general different from $\f(x,g).(n_1,n_2)=(x.n_1,x.g.n_2)$.
	Anyhow, to show that $\psi_N$ is surjective, it suffices to show that $\psi(f)\in (k[X]\otimes k[G])^{N\times N}$ for $f\in k[X]\otimes k[X]$ implies $f\in (k[X]\otimes k[X])^{N\times N}$. But $\psi(f)\in (k[X]\otimes k[G])^{N\times N}$ means that 
	\begin{equation}\label{eq: f invariant}
		f(x,x.g)=\psi(f)(x,g)=\psi(f)(x.n_1,g.n_2)=f(x.n_1,x.n_1gn_2)
	\end{equation}
	for $n_1,n_2\in N(R)$, $x\in X(R')$, $g\in G(R')$, $R$ a $k$-algebra and $R'$ an $R$-algebra.
	
	Given a $k$-algebra $\widetilde{R}$, $\widetilde{n}_1, \widetilde{n}_2\in N(\widetilde{R})$, an $\widetilde{R}$-algebra $\widetilde{R}'$ and $x_1,x_2\in X(\widetilde{R}')$, we can write $x_2=x_1.\widetilde{g}$ for a unique $\widetilde{g}\in G(\widetilde{R}')$. 
	Then, using (\ref{eq: f invariant}) with $R=R'=\widetilde{R}'$, $x=x_1$, $g=\widetilde{g}$, $n_1=\widetilde{n}_1$,  $n_2=\widetilde{g}^{-1}\widetilde{n}_1^{-1}\widetilde{g}\widetilde{n}_2\in N(\widetilde{R}')$, we have $$f(x_1,x_2)=f(x_1,x_1.\widetilde{g})=f(x_1.\widetilde{n}_1, x_1.\widetilde{n}_1\widetilde{g}\widetilde{g}^{-1}\widetilde{n}_1^{-1}\widetilde{g}\widetilde{n}_2)=f(x_1\widetilde{n}_2,x_1.\widetilde{g}\widetilde{n}_2)=f(x_1.\widetilde{n}_1,x_2.\widetilde{n}_2).$$
	Thus $f\in (k[X]\otimes k[X])^{N\times N}=k[X]^N\otimes k[X]^N$ as desired and we can conclude that $X_N$ is a $G/N$-torsor. Since $G/N$ is algebraic, also $X_N$ has to be of finite type. Indeed, a $K$-point of $X_N$ in some field extension $K$ of $k$, yields an isomorphism of $K$-algebras between $k[G/N]\otimes K$ and $k[X_N]\otimes K$. As $k[G/N]$ is a finitely generated $k$-algebra we see that $k[G/N]\otimes K$ and therefore also $k[X_N]\otimes K$ are finitely generated $K$-algebras. But then $k[X_N]$ has to be a finitely generated $k$-algebra.
	
	We next show that $k[X]$ is the directed union of the $k[X]^N$'s. Because each $k[X]^N$ is a finitely generated $k$-algebra, it suffices to show that every finite subset $F$ of $k[X]$ is contained in some $k[X]^N$. 
	
	Note that $\rho\colon k[X]\to k[X]\otimes k[G]$ defines the structure of a (right) comodule on $k[X]$.  According to \cite[Theorem~3.3]{Waterhouse:IntroductiontoAffineGroupSchemes}, every comodule is the directed union of its finite dimensional (as $k$\=/vector space) comodules. So $F$ is contained in a finite dimension $k$-subspace $V$ of $k[X]$ such that $\rho(V)\subseteq V\otimes k[G]$. Let $A$ be the $k$-subalgebra of $k[X]$ generated by $V$. Then $A$ is finitely generated and $\rho(A)\subseteq A\otimes k[G]$. In fact, since $A$ is finitely generated, there exists a finitely generated $k$-subalgebra $B$ of $k[X]$ such that $\rho(A)\subseteq B\otimes A$. According to \cite[Section 3.3]{Waterhouse:IntroductiontoAffineGroupSchemes} every Hopf algebra is the directed union of Hopf subalgebras that are finitely generated as $k$-algebras. Thus $B$ is contained in some Hopf subalgebra 
	$B'$ of $k[G]$ that is finitely generated as a $k$-algebra. Any Hopf subalgebra of $k[G]$ is of the form $k[G/N]=k[G]^N$ for a normal closed subgroup $N$ of $G$ (\cite[Sections 15 and 16]{Waterhouse:IntroductiontoAffineGroupSchemes} or \cite[Theorem 4.3]{Takeuchi:ACorrespondenceBetweenHopfidealsAndSubHopfalgebras}). Thus $B'=k[G]^N$ for a normal closed subgroup $N$ of $G$ with $G/N$ algebraic. Moreover, $F\subseteq A$ and $\rho(A)\subseteq A\otimes k[G]^N$. It thus suffices to show that $A\subseteq k[X]^N$. For $f\in A\subseteq k[X]$ we have $\rho(f)\in k[X]\otimes k[G]^N$. This mean that for a $k$-algebra $R$, an $R$-algebra $R'$, $n\in N(R)$, $g\in G(R')$ and $x\in X(R')$ we have $f(x.g)=f(x.gn)$. 
	Choosing $g=1$, we see that $f\in k[X]^N$. So $A\subseteq k[X]^N$ as desired.
	
	Note that $k[X]^N\subseteq k[X]^{N'}$ if $N'\subseteq N$.
	Since $k[X]$ is the directed union of the $k[X]^N$'s, we see that $X=\varprojlim X_N$, where the projective limit is taken over the set of all closed normal subgroups $N$ of $G$ such that $G/N$ is algebraic. This index set is a directed set with respect to the partial order defined by $N\leq N'$ if $N'\subseteq N$.
	
	To finish the proof it remains to verify that $C_G(X_N)=N$. Since the action of $G$ on $X_N$ factors through $G/N$, surely $N\subseteq C_G(X_N)$. Conversely, if $R$ is a $k$-algebra and $g\in C_G(X_N)(R)$, then the image $\overline{g}$ of $g$ in $(G/N)(R)$ acts trivial on $X_N(R')$ for every $R$-algebra $R'$. Since $X_N$ is an $G/N$-torsor we must have $\overline{g}=1$, i.e., $g\in N(R)$. Thus $C_G(X_N)\subseteq N$ and consequently $C_G(X_N)=N$.	
	%	$\psi((k[X]\otimes k[X])^{N\times N})\subseteq k[X]\otimes k[G]$ consists of those $f\in k[X]\otimes k[G]$ such that 
%	$f(x.n_1,x.n_1gn_2)=f(x,g)$
\end{proof}

The following lemma introduces the orbit topology on $X(k)$, where $X$ is a $G$-space such that $X$ is a $G/C_G(X)$-torsor. This topology is similar to the topology on the $k$-points of an affine algebraic group discussed before Proposition 2.8 in \cite{HochschildMostow:RepresentationsAndRepresentativeFunctionsOfLieGroups}.

\begin{lemma} \label{lemma: coset topology}
	Assume that $k$ is algebraically closed and let $G$ be an affine group scheme.
	
	\begin{enumerate}
		\item  Let $X$ be an affine $G$-space of finite type such that $X$ is a $G/C_G(X)$-torsor. Then the subsets of $X(k)$ that are finite unions of orbits of the form $x.H(k)$ with $x\in X(k)$ and $H$ a closed subgroup of $G$, are the closed subsets for a topology on $X(k)$; the \emph{orbit topology}. With respect to the orbit topology $X(k)$ is a compact T1 space.  
		\item Let $\f\colon X_2\to X_1$ be a morphism of affine $G$-spaces of finite type such that $X_i$ is a $G/C_G(X_i)$-torsor ($i=1,2$). Then the map $\f_k\colon X_2(k)\to X_1(k)$ is continuous and closed with respect to the orbit topologies.
	\end{enumerate}
\end{lemma}
\begin{proof}
	For (i), we first show that an orbit of the form $x.H(k)$ with $x\in X(k)$ and $H$ a closed subgroup of $G$ is a closed subset of $X(k)$ with respect to the Zariski topology.
	
	Set $G'=G/C_G(X)$ and let $H'$ denote the image of $H$ in $G'$. Then $H\to H'$ is a quotient map and by \cite[Chapter III, Cor. 7.6]{DemazureGabriel:GroupesAlgebriques} the map $H(k)\to H'(k)$ is surjective. Thus $x.H(k)=x.H'(k)$. As $X$ is a $G'$-torsor, the morphism $G'\to X,\ g'\mapsto x.g'$ is an isomorphism. In particular, $G'(k)\to X(k)$ is a homeomorphism mapping the closed subset $H'(k)$ to the closed subset $x.H'(k)$. So $x.H(k)$ is closed with respect to the Zariski topology and so is any finite union of such orbits.
	
	Since $X$ is of finite type, any descending chain of Zariski closed subsets of $X$ is finite. Thus an arbitrary intersection of finite unions of orbits is in fact a finite intersection of finite unions of orbits. 
	Therefore, to show that an arbitrary intersection of finite unions of orbits is itself a finite union of orbits, it suffices to show that the intersection of two orbits is again an orbit. So let $H_1,H_2$ be closed subgroups of $G$ and $x_1, x_2\in X(k)$. If $(x_1.H_1(k))\cap (x_2.H_2(k))$ is non-empty, then there exists an $x\in X(k)$ such that $x_1H_1(k)=xH_1(k)$ and $x_2H_2(k)=x H_2(k)$. Moreover, as noted above, we have $xH_1(k)=xH_1'(k)$ and $xH_2'(k)$ with $H_i'$ the image of $H_i$ in $G'$. Then
	$$(x_1.H_1(k))\cap (x_2.H_2(k))=(x.H'_1(k))\cap (x.H'_2(k)=x.(H'_1(k)\cap H'_2(k))=x.(H_1'\cap H_2')(k), $$
	where the second equality uses that $G'(k)$ acts freely on $X(k)$. Thus, if $H\leq G$ denotes the inverse image of $H_1'\cap H_2'\leq G'$ under the quotient map $G\to G'$, then  $(x_1.H_1(k))\cap (x_2.H_2(k))=x.H(k)$.
	
	Therefore the finite unions of orbits are indeed the closed sets of a topology on $X(k)$. As noted above, a subset of $X(k)$ that is closed with respect to the orbit topology is closed with respect to the Zariksi topology. In particular, any descending chain of closed subsets with respect to the orbit topology is finite. Hence $X(k)$ is compact with respect to the orbit topolgy. The points of $X(k)$ are closed with respect to the orbit topology because they are the orbits of the trivial subgroup $H=1$ of $G$. This concludes the proof of (i).
	
	For (ii), we first show that $\f_k\colon X_2(k)\to X_1(k)$ is surjective. 
	Let $x_1\in X_1(k)$. The group $G(k)$ acts transitively on $X_1(k)$ because $G(k)\to (G/C_G(X_1))(k)$ is surjective (again by \cite[Chapter III, Cor. 7.6]{DemazureGabriel:GroupesAlgebriques}) and $X_1(k)$ is a $(G/C_G(X_1))(k)$-torsor. Thus, if $x_2$ is any element of $X_2(k)$, there exists a $g\in G(k)$ such that $x_1=\f_k(x_2).g=\f_k(x_2.g)$. Hence $\f_k$ is surjective.
	
	To show that $\f_k$ is continuous with respect to the orbit topologies, it suffices to show that the inverse image of an orbit is on orbit. So let $H$ be a closed subgroup of $G$ and $x_1\in X_1(k)$. We would like to show that $\f_k^{-1}(x_1.H(k))$ is an orbit. As noted in the proof of (i), we have $x_1.H(k)=x_1.H'(k)$, where $H'$ denotes the image of $H$ in $G/C_G(X_1)$. In other words, we may assume that $C_G(X_1)\leq H$.
	Since $\f_k$ is surjective, there exists an $x_2\in X_2(k)$ such that $\f_k(x_2)=x_1$. We claim that $\f_k^{-1}(x_1.H(k))=x_2.H(k)$. Clearly, $x_2.H(k)\subseteq \f_k^{-1}(x_1.H(k))$. For the reverse inclusion, let $x_2'\in \f_k^{-1}(x_1.H(k))$. Since $G(k)$ acts transitively on $X_2(k)$, there exists a $g\in G(k)$ such that $x'_2=x_2.g$. Then $$x_1.g=\f_k(x_2).g=\f_k(x_2.g)=\f_k(x'_2)\in x_1.H(k).$$
	Hence there exists an $h\in H(k)$ such that $x_1.g=x_1.h$. Since $X_1(k)$ is a $(G/C_G(X_1))(k)$-torsor, we have $gh^{-1}\in C_G(X_1)(k)$. But $C_G(X_1)\leq H$ and so $g\in H(k)$. Thus $x_2'=x_2.g\in x_2.H(k)$. Therefore $\f_k^{-1}(x_1.H(k))=x_2.H(k)$ and $\f_k$ is continuous with respect to the orbit topologies. 
	
	To see that $\f_k$ is closed with respect to the orbit topologies, it suffices to see that $\f_k$ preserves orbits. But this follows immediately from the $G(k)$-equivariance of $\f_k$.
\end{proof}

%\begin{theo} \label{theo: main}
%	Let $k$ be an algebraically closed field, $G$ an affine group scheme over $k$ and $X$ a $G$-torsor over $k$. Then $X$ is trivial, i.e., $X(k)\neq\emptyset$.
%\end{theo}

We are now prepared to prove our main results.

\begin{proof}[Proof of Theorem \ref{theo: main intro}]
	We first note that $X$ is an affine scheme. Indeed, $X$ and $G$ become isomorphic over some field extension $K$ of $k$. So $X_K$ is an affine scheme. By faithfully flat descent, $X$ is an affine scheme (\cite[Expos\'{e} VIII, Cor. 5.6]{Grothendieck:SGA1}).

	By Proposition \ref{prop: torsor is projective limit} we may write $X$ as a projective limit $X=\varprojlim_{i\in I} X_i$ of affine $G$-spaces $X_i$ of finite type such that each $X_i$ is a $G/C_G(X_i)$-torsor.
	In particular, $X(k)=\varprojlim_{i\in I} X_i(k)$.
	
	By Lemma \ref{lemma: coset topology} each $X_i(k)$ is a compact T1 space with respect to the orbit topology. Moreover, the transition maps $X_j(k)\to X_i(k)$ ($j\geq i$) are continuous and closed with respect to the orbit topologies. Thus Lemma \ref{lemma: projective limit nonempty} applied to the projective system of the $X_i(k)$'s equipped with the orbit topology shows that $X(k)$ is non-empty.
\end{proof}

\begin{proof}[Proof of Corollary \ref{cor: intro uniquenes of fibre functor}]
	Let $\omega_1,\omega_2$ be two neutral fibre functors on a neutral tannakian category over an algebraically closed field $k$. Then $G=\underline{\Aut}^\otimes(\omega_1)$ is an affine group scheme and $\underline{\operatorname{Isom}}(\omega_1,\omega_2)$ is a $G$-torsor  (\cite[Theorem 3.2]{DeligneMilne:TannakianCategories}). By Theorem \ref{theo: main intro} the $G$-torsor $\underline{\operatorname{Isom}}(\omega_1,\omega_2)$ has a $k$-point, i.e., $\omega_1$ and $\omega_2$ are isomorphic.
\end{proof}

\bibliographystyle{alpha}
 \bibliography{bibdata}
\end{document}